\documentclass[11pt,draft,
reqno]{amsart}

\usepackage{amsmath,amssymb,amscd,amsfonts,amsthm}
\usepackage{dsfont}
\usepackage[all]{xy}
\usepackage{enumerate}

{\catcode`\@=11
\gdef\n@te#1#2{\leavevmode\vadjust{%
 {\setbox\z@\hbox to\z@{\strut#1}%
  \setbox\z@\hbox{\raise\dp\strutbox\box\z@}\ht\z@=\z@\dp\z@=\z@%
  #2\box\z@}}}
\gdef\leftnote#1{\n@te{\hss#1\quad}{}}
\gdef\rightnote#1{\n@te{\quad\kern-\leftskip#1\hss}{\moveright\hsize}}
\gdef\?{\FN@\qumark}
\gdef\qumark{\ifx\next"\DN@"##1"{\leftnote{\rm##1}}\else
 \DN@{\leftnote{\rm??}}\fi{\rm??}\next@}}

\DeclareFontFamily{OT1}{wncyr}{\hyphenchar\font45 }
\DeclareFontShape{OT1}{wncyr}{m}{n}{%
   <5> <6> <7> <8> <9> gen * wncyr
   <10> <10.95> <12> <14.4> <17.28> <20.74>  <24.88>wncyr10}{}
\DeclareFontShape{OT1}{wncyr}{m}{it}{%
   <5> <6> <7> <8> <9> gen * wncyi
   <10> <10.95> <12> <14.4> <17.28> <20.74> <24.88> wncyi10}{}
\DeclareFontShape{OT1}{wncyr}{m}{sc}{%
   <5> <6> <7> <8> <9> <10> <10.95> <12> <14.4>
   <17.28> <20.74> <24.88>wncysc10}{}
\DeclareFontShape{OT1}{wncyr}{b}{n}{%
   <5> <6> <7> <8> <9> gen * wncyb
   <10> <10.95> <12> <14.4> <17.28> <20.74> <24.88>wncyb10}{}
\input cyracc.def

\DeclareMathSizes{9}{9}{7}{5}

\theoremstyle{plain}

\newtheorem{theorem}{Theorem}

\newtheorem{proposition}[theorem]{Proposition}
\newtheorem{lemma}[theorem]{Lemma}
\newtheorem{corollary}[theorem]{Corollary}

\theoremstyle{definition}

\newtheorem{problem}[theorem]{Problem}

\newtheorem{definition}[theorem]{Definition}
\newtheorem{remark}[theorem]{Remark}

\newtheorem{nothing*}[theorem]{}
\newtheorem{subnothing*}[sub]{}
\newtheorem{example}[theorem]{Example}

\newtheorem{question}[theorem]{Question}

\theoremstyle{remark}

\newcommand{\lb}{{\rm(\hskip -.1mm}}
\newcommand{\rb}{{\hskip .13mm\rm)}}

\begin{document}

\title[Variations on the theme of Zariski's Cancellation Problem]
{Variations on the theme of\\ Zariski's Cancellation Problem}

\author[Vladimir L. Popov]{Vladimir L. Popov${}^{1}$}
\thanks{${}^1$  Steklov Mathematical Institute,
Russian Academy of Sciences, Gub\-kina 8, Moscow
119991, Russia. {\it E-mail address}:  popovvl@mi.ras.ru}

\begin{abstract}
This  is an expanded version of
the talk by the author at the confe\-ren\-ce\;{\it Poly\-no\-mi\-al Rings and Affine Algebraic Geometry},
February 12--16, 2018, Tokyo Me\-tropolitan University, Tokyo, Ja\-pan. Considering
a local ver\-sion of the Zariski Cancellation Prob\-lem naturally leads
to explo\-ration of some classes of varieties of special kind and their
equivariant versions.\;We discuss several to\-pics in\-spi\-red by this
exploration, including the problem of classify\-ing a class of
affine algebraic groups that are naturally singled out in
studying the conjugacy problem for algebraic subgroups of the Cremona groups.
\end{abstract}

\maketitle

\subsection*{1. Introduction}
This is an expanded version of
the
talk
by the author
at the conference {\it Polyno\-mial Rings and Affine Algebraic Geometry},
February 12--16, 2018, Tokyo Metro\-politan University, Tokyo, Japan.

Our starting point is a local version of the Zariski Cancellation Problem
(LZCP). Its consideration naturally leads to distinguishing a class of
varieties of a special kind, called here flattenable, and a more general
class of  locally flattenable varieties. We discuss the relevant examples,
including flattenabi\-lity of affine algebraic groups and the related
varieties, in particular, we prove that all smooth spherical varieties
are locally flattenable. This is completed by answering (LZCP). We
then consider the equivariant versions of flattenability and obtain a
series of results on equivariant flattenability of affine algebraic groups.
In particular, we prove that a reductive algebraic group is equivariantly
flattenable if and only if it is linearly equivariantly flattenable, and we
prove that equivariant flattenability of a Levi subgroup of a connected
affine algebraic group $G$ implies that of $G$. The latter  yields
that every connected solvable affine algebraic group is equivariantly
flattenable. As an application, we briefly survey
a special role of equivariantly flattenable subgroups in the rational
linearization problem. Then we dwell on the classi\-fication problem
of equivariantly flattenable affine algebraic groups $G$. We prove that
every such $G$ is special in the sense of Serre, which implies that
if $G$ is reductive equivariantly flattenable, then its derived group if
a product of the groups of types ${\mathrm {SL}}$ and ${\rm Sp}$.
We complete this discussion with the unexpected recent examples
of reductive equivariantly flattenable groups, whose derived groups do contain
factors of type ${\rm Sp}$. In the last section, the local version of
equivariantly flattenable
varieties and the relevant version of the Gromov problem are briefly
considered.

\vskip 2mm
\noindent
{\it Acknowledgement.} I am grateful to the referee for thoughtful reading
and sug\-gestions.

\subsubsection*{\it Notation, Conventions, and Terminology}
We fix an algebraically closed field $k$ of characteristic zero.
In what follows, as in
 \cite{Bo1991}, \cite{PV1994},
 variety means algebraic variety over $k$ in the sense of Serre
(so algebraic group means
algebraic group over $k$). Unless otherwise specified, all topolo\-gical
terms refer to the Zariski topology. We use freely the standard
notation and conventions of
loc. cit., where the proofs of the unreferenced claims
 and/or the relevant references can be found. Action
 of an algebraic group on an algebraic
variety means algebraic (morphic) left action;  homomorphism of
an algebraic group means  algebraic homomorphism.

Recall that a {\it Levi subgroup} of a connected affine
algebraic group $G$ is its (necessarily connected reductive) subgroup
$L$ such that $G$ is the semi-direct product of $L$ and
the unipotent radical ${\mathcal R}_uG$ of $G$; since
${\rm char}(k)=0$,
by  a result of G. D. Mostow, Levi subgroups exist and are conjugate
in $G$ (cf.\;\cite[11.22]{Bo1991}).

We also use the following notation:
\begin{enumerate}[\hskip 4.2mm $\cdot$]
\item $A^*$ is the group of units of an associative $k$-algebra $A$ with identity.

\item ${\rm Mat}_{n\times m}$ is the $k$-vector space of all
$n\times m$-matrices with entries in $k$; for $n=m$, it is naturally endowed with
the $k$-algebra structure.

\item $A^{\top}$ is the transpose of a matrix $A\in {\rm Mat}_{n\times m}$.
\end{enumerate}

\subsection*{2. The Zariski Cancellation Problem %%(ZCP)
} So is called
the following qu\-estion:
\begin{equation}\label{ZCP}
\begin{split}
&\mbox{\it Are there affine varieties $X$ and $Y$ such that
$Y$ and  } \\[-1.5mm]
&\mbox{\it $X\times Y$ are isomorphic respectively to
${\mathbf A}^{\hskip -.4mm s}$ and $
{\mathbf A}^{\hskip -.4mm s+d}$,  }\\[-1.5mm]
&\mbox{\it but
$X$ is not isomorphic to
${\mathbf A}^{\hskip -.4mm d}$?
}
\end{split}
\tag{ZCP}
\end{equation}

\vskip 2mm\noindent
At this writing (January 2019), for $d>2$, it is still open.
A historical survey about
the Zariski Cancellation Problem is given in \cite{Gu2015}.

Our starting point is a local version of this problem.
Making precise its formulation (see \eqref{LZCP} in Subsection 5)
leads to distinguishing the following class of varieties:
\begin{definition}\label{def1} An irreducible variety $X$ is called
\begin{enumerate}[\hskip 4.2mm\rm(a)]
\item {\it flattenable} if $X$ isomorphic to an open subset
of an affine space;
\item {\it locally flattenable} if for every point  $x\in X$
there is a flattenable open subset of $X$ containing $x$.
\end{enumerate}
\end{definition}

\subsection*{3. Terminology} Under other names, locally
flattenable varieties  appeared in the literature long ago.
The earliest  reference known to the author is \cite[p. 2-09]{Ch1958}
where Chevalley calls them {\it special} varieties.
In \cite{Ma1974} Chevalley terminology is used for
the definition of $R$-equivalence.\,In \cite{Ak1993} these
varieties appear as {\it algebraic spaces}, in \cite{BHSV2008}
as {\it plain varieties}, and in \cite{BB2014} and \cite{Pe2017}
as {\it uniformly rational varieties}. The term {\it flattenable variety}
is coined in \cite{Po2013_1}, where  special properties of
linearly equivariantly flat\-ten\-able algebraic subgroups
(see below Definition \ref{def3}) of the Cremona groups
have been revealed (this topic is briefly surveyed
in Subsection 9 below).

By Definition \ref{def1}, every locally flattenable variety
is rational. Whether the converse is true is open
at this writing (January 2019):
\begin{equation}\label{Gr}
\mbox{\it Is every irreducible smooth rational variety locally flattenable?}
\tag{Gr}
\end{equation}
This problem  was raised
by M. Gromov in \cite[$3.5.{\rm E}'''$]{Gr1989}
(for projective varieties).

\subsection*{4. Examples of Locally Flattenable Varieties}

\subsubsection*{\it {\rm 1.} Homogeneous Spaces}

\begin{theorem}Let $X$ be an irreducible variety. If the natural
action of ${\rm Aut}(X)$ on $X$ is transitive, then the following
properties are equivalent:
\begin{enumerate}[\hskip 4.2mm\rm(a)]
\item $X$ is a rational variety;
\item $X$ is a locally flattenable variety.
\end{enumerate}
\end{theorem}
\begin{proof} If (a) holds, then $X$ contains an open flattenable
subset $U$. Since $U$ and $gU$ for any $g\in {\rm Aut}(X)$ are
isomorphic, (b) follows from the equality $X=\bigcup_{g\in G}gU$
(the latter holds because of the transitivity condition).
Defi\-nition \ref{def1} implies (b)$\Rightarrow$(a).
\quad $\square$ \end{proof}

\begin{corollary}\label{hs} Let $G$ be a connected affine algebraic
group and let $H$ be a closed subgroup of $G$. Then the following
properties are equivalent:
\begin{enumerate}[\hskip 4.2mm\rm(a)]
\item $G/H$ is a rational variety;
\item $G/H$ is a locally flattenable variety.
\end{enumerate}
\end{corollary}

\begin{remark} Maintain the notation of Corollary \ref{hs}.
There are nonrational (and even not stably rational) varieties
$G/H$, where  $G$ is
\begin{equation*}
%%\label{ssspec}
{\mathrm {SL}}_{n_1}\times\cdots
\times {\mathrm {SL}}_{n_r}\times {\rm Sp}_{2m_1}\times
\cdots \times {\rm Sp}_{2m_s}
\end{equation*}
 and $H$ is finite; see
\cite[Thm. 2]{Po2013_2}. It is an old problem,  still open at this
writing (October 2019), whether there are nonrational homogeneous
spaces $G/H$ with  connected $H$. For connected $H$ of various
special types, rationality of $G/H$ is known; see \cite{CZ2017}
and Remark \ref{ratio} below. In particular, $G$ is rational as a variety
\cite{Ch1954} (cf. \cite[Lem. 2]{Po2013_2}).
\end{remark}

\subsubsection*{\it {\rm 2.} Vector Bundles and Homogeneous Fiber Spaces}\

\begin{definition}\label{deff2}
{\rm Given three varieties $X$, $Z$, $F$, we say that a surjective morphism $\varphi\colon X\to Z$ is a {\it locally trivial fibration over $Z$ with fiber} $F$ if for every point $z\in Z$, there are a neighbourhood $U$ of $z$ in $Z$ (called {\it trivializing neighborhood}) and an isomorphism $\tau_U\colon \varphi^{-1}(U)\to U\times F$ over $U$. If $F$ has a structure of an algebraic vector space over $k$ and, for any pair of trivializing neighborhoods $U$ and $V$, the automorphism $\tau_V\circ\tau_U^{-1}\colon (U\cap V)\times F\to (U\cap V)\times F$ over $U\cap V$
%%defined by $\tau_U$ and $\tau_V$
is linear over every point of $U\cap V$,
%%
%%for any pair of trivializing neighborhoods $U$ and $V$ and any point $z\in U\cap V$, the %%natural automorphism $z\times F\to z\times F$ over $z$ defined by $\tau_U$ and %%$\tau_V$ is linear,
then we say that $\varphi$ is {\it a vector bundle over $Z$ with fiber $F$}.}
\end{definition}

The claim of Theorem 2 below for vector bundles is mentioned in \cite[Ex. 2.1]{BB2014}:

\begin{theorem}\label{thm4n}
Let $X\to Z$ be a locally trivial fibration over an irreducible variety $Z$ with fiber
${\mathbf A}\!^n$. If $Z$ is locally flattenable,
then $X$  is  locally flattenable as well.
\end{theorem}

\begin{proof}
%%In view of
This follows from  Definitions \ref{def1} and \ref{deff2}.
%%, this follows from the above definition of
%%a locally trivial fibration.
\quad $\square$
\end{proof}

We recall (see \cite[3.2]{Se1958}, \cite[4.8]{PV1994}) a construction used several times below.
Let $G$ be a connected algebraic group, $H$ its closed subgroup, and
$F$ a variety endowed with an
action of $H$. Then $H$ acts on $G\times F$ by the formula
\begin{equation*}
h\cdot (g, f)\mapsto (gh^{-1}, h\cdot f).
\end{equation*}
By \cite[Prop. 4]{Se1958}, a mild restriction on $F$ ensures the existence of the quotient of this action (in the sense of \cite[6.3]{Bo1961}): namely, it exists if every finite subset of $F$ lies in affine open subset of $F$ (for instance, any quasiprojective $F$ shares this property).
The corresponding quotient variety is denoted by $G\times^HF$. The natural projection $G\times F\to G$ is $H$-equivariant and therefore induces the surjective morphism  of the quotients $\pi_{G, H, F}^{\ }\colon G\times^HF\to G/H$; its fibers are isomorphic to $F$. The action of $G$ on $G\times F$ by left multiplications on the first factor
commutes with the action of $H$ and therefore descends to the action of $G$ on $G\times^HF$; the morphism $\pi_{G, H, F}^{\ }$ is equivariant with respect to this action and the natural action of $G$ on $G/H$. Given the aforesaid,
$G\times^HF$ is called the {\it $($algebraic$)$ homogeneous fiber space over $G/H$ with fiber $F$}. In general, $\pi_{G, H, F}^{\ }$ is not a locally trivial fibration over $G/H$ with fiber $F$ (in the sense of  Definition \ref{deff2}). However, if $F$ is a vector space $V$ over $k$ and the action of $H$ on $V$ is linear, then \cite[Thm.\;2]{Se1958} implies that $\pi_{G, H, F}^{\ }$ is a vector bundle over
$G/H$ with fiber $V$ (in the sense of  Definition \ref{deff2}).

%%\vskip 5mm

%%Then we have (see \cite[4.8]{PV1994}) the algebraic homogeneous
%%fiber space $G\times^HF$ over $G/H$ with fiber $F$; the natural
%%projection
%%$\pi_{G, H, F}^{\ }\colon G\times^HF\to G/H$ is locally trivial
%%in the \'etale topology. If $F$ is a vector space $V$ over $k$ and
%%the action of $H$ on $V$ is linear,  then $\pi_{G, H, V}^{\ }$
%%is an algebraic vector bundle over $G/H$ with fiber $V$.
%%Combining Corollary \ref{hs} and Theorem \ref{thm4} yields

%%\vskip 30mm

%%The claim of  Theorem \ref{thm4} below is mentioned in \cite[Ex. 2.1]{BB2014}:

%%\begin{theorem}\label{thm4} Let $X\to Z$ be an {\rm(}algebraic{\rm)} vector
%%bundle over an irreducible variety $Z$. If $Z$ is locally flattenable,
%%then $X$  is  locally flattenable as well.
%%\end{theorem}
%%\begin{proof} Since the fibers of $X\to Z$ are isomorphic to an
%%affine space, and,
%%by \cite[Thm.\;2]{Se1958}, algebraic vector bundles are locally
%%trivial in the Zariski topology,
%%the claim follows from Definition \ref{def1}.
%%\quad $\square$ \end{proof}

%%Let $G$ be a connected algebraic group, $H$ its closed subgroup, and
%%$F$ a quasiprojective variety endowed with a regular action of $H$.
%%Then we have (see \cite[4.8]{PV1994}) the algebraic homogeneous
%%fiber space $G\times^HF$ over $G/H$ with fiber $F$; the natural
%%projection
%%$\pi_{G, H, F}^{\ }\colon G\times^HF\to G/H$ is locally trivial
%%in the \'etale topology. If $F$ is a vector space $V$ over $k$ and
%%the action of $H$ on $V$ is linear,  then $\pi_{G, H, V}^{\ }$
%%is an algebraic vector bundle over $G/H$ with fiber $V$.

Combining Corollary \ref{hs} and Theorem \ref{thm4n} yields

\begin{corollary}\label{cor5}
Maintain the above notation. If $ G/H$ is rational, then
$G\times^HV$ is locally flattenable.
\end{corollary}

\begin{theorem}\label{thm7} Let $G$ be a connected reductive
algebraic group and let $X$ be a
smooth affine variety endowed with an action of $G$. Assume that
\begin{enumerate}[\hskip 4.2mm\rm(a)]
\item $k[X]^G=k$;
\item the {\rm(}unique{\rm, see, e.g., {\rm \cite[Cor. of Thm. 4.7]{PV1994})}}
closed $G$-orbit $\mathcal O$ in $X$ is rational.
\end{enumerate}
Then $X$ is locally flattenable.
\end{theorem}

\begin{proof}
By \cite[p. 98, Cor. 2]{Lu1973} (see also \cite[Thm. 6.7]{PV1994}),
(a) and smoothness of $X$ imply that $X$ is  $G$-equivariantly isomorphic
to $G\times^HV$, where $H$ is the $G$-stabilizer of a point of $\mathcal O$,
and $V$ is a finite-dimensional $H$-module. The claim then follows from Corollary \ref{cor5}.
\quad $\square$ \end{proof}

\subsubsection*{\it {\rm 3}. Spherical Varieties}\

Let $G$ be a connected reductive algebraic group and let $B$ be a
Borel subgroup of $G$. Recall that a variety $X$ endowed with an
action of $G$ is called {\it spherical variety of $G$} if there is a
dense open $B$-orbit in $X$.

\begin{theorem}\label{sph} Every smooth spherical variety is locally
flattebable.
\end{theorem}
\begin{proof} Let $X$ be a smooth spherical variety of a connected
reductive group $G$.

First, $X$ is rational because every $B$-orbit is rational (the latter
is isomor\-phic to
the complement of a union of several coordinate hyperplanes in some
affine space \cite[Cop. p.5-02]{Gr1958}).

 Secondly, every $G$-orbit in $X$ is spherical (see, e.g.,
 \cite[Prop. 15.14]{Ti2011}),
 hence  rational. Therefore, by Theorem \ref{thm7}, if $X$ is
 affine, then $X$ is locally flattenable.

 Thirdly, arbitrary $X$ is covered by open subsets, each of which is isomor\-phic
 to a variety of the form $P\times^LZ$, where $P$ and $L$ are respectively
 a parabolic subgroups of $G$ and  a Levi subgroup of $P$, and $Z$ is
 an affine spherical variety of $L$; see, e.g., \cite[Thm. 15.17]{Ti2011}.
 Since $X$ is smooth, $Z$ is smooth as well. Therefore,
 as explained above, $Z$ is locally flattenable. The variety $P\times^LZ$
 is isomorphic to the product of $Z$ and
 the underlying variety of
 the unipotent radical of $P$. Since this underlying variety is isomorphic
 to an affine space
 \cite[Cor.\,p.\,5-02]{Gr1958}, we infer that $P\times^LZ$ is locally flattenable.
 Therefore, $X$ is locally flattenable, too.
 \quad $\square$ \end{proof}

Since every toric variety is spherical, Theorem \ref{sph} implies

\begin{corollary}[{{\rm
\cite[Expl. 2.2]{BB2014}}}]\label{toric} Every smooth toric
variety is locally flattenable.
\end{corollary}

\subsubsection*{\it {\rm 4}. Blow-ups with Nonsingular Centers}

\begin{theorem}[{{\rm
\cite[p. 885, Prop.]{Gr1989}, \cite[Thm. 4.4]{BHSV2008},
\cite[Prop. 2.6]{BB2014}}}]
\label{blow} The blow-up of a locally flattenable variety along
a smooth subvariety is locally flattenable.
\end{theorem}

\subsubsection*{\it {\rm 5}. Curves and Surfaces}\

For varieties of dimension $\leqslant 2$, the answer to
\eqref{Gr} is affirmative:

\begin{theorem}[{{\rm \cite[Prop. 3.2]{BHSV2008},
\cite[Prop. 2.6]{BB2014}}}] Every irredu\-cible rational
smooth algebraic curve
or surface $X$ is locally flattenable.
\end{theorem}
\begin{proof} If $X$ is a curve, it admits an open embedding
in $\mathbf P^1$. If $X$ is a surface, it
admits an open embedding in a projective smooth  surface,
which, being rational,  is obtained by repeated point blow-ups of
a minimal model, i.e., either ${\mathbf P}^2$ or a Hirzebruch surface
$F_n$, $n\neq 1$.
Since ${\mathbf P}^1$, ${\mathbf P}^2$, and $F_n$
are toric varieties, the claim follows from Corollary \ref{toric} and
Theorem \ref{blow}.
\quad $\square$ \end{proof}

\subsection*{5. Local Version of (ZCP)} Given Definition \ref{def1},
the local version of  the Zariski Cancellation Prob\-lem men\-tion\-ed
in Subsection 2 is formulated as fol\-lows:
\begin{equation}\label{LZCP}
\begin{split}
&\mbox{\it Are there affine varieties $X$ and $Y$ such that
$Y$ and } \\[-1.5mm]
&\mbox{\it $X\times Y$ are flattenable, but
$X$ is not flattenable?
}
\end{split}
\tag{LZCP}
\end{equation}

\noindent In Subsection 7 we show that the answer to
\eqref{LZCP} is affirmative.

\subsection*{6. Flattenable Varieties vs. Locally Flattenable Varieties}
Flat\-te\-nab\-le varieties have special properties:
\begin{lemma}\label{thm12}
Let $X$ be an affine flattenable variety and let $\varphi\colon
X\hookrightarrow {\mathbf  A}^{\hskip -.4mm n}$ be an open
embedding.\;If $k[X]^\ast=k^\ast$, then $\varphi(X)=
{\mathbf  A}^{\hskip -.4mm n}$.
\end{lemma}
\begin{proof}
Assume that the closed set ${\mathbf  A}^{\hskip -.4mm n}\setminus \varphi(X)$
is nonempty.\;Then, since $X$ is affine, the dimension
of every irreducible component  of
this set is $n-1$.\;Therefore, ${\rm Pic}({\mathbf  A}^{\hskip -.4mm n})=0$
implies that ${\mathbf  A}^{\hskip -.4mm n}\setminus \varphi(X)$ is
the set of zeros of a certain function $f\in k[{\mathbf  A}^{\hskip -.4mm n}]$.\;Then
$f\circ\varphi$ is a nonconstant element of $k[X]^\ast$,
a contradiction.\;Hence
$\varphi(X) ={\mathbf  A}^{\hskip -.4mm n}$.
\quad $\square$ \end{proof}

\begin{lemma}\label{afsp} For a connected affine algebraic group $G$,
the following proper\-ties are equivalent:
\begin{enumerate}[\hskip 4.2mm \rm(a)]
\item as a variety, $G$ is isomorphic to an affine space;
\item as a group, $G$ is unipotent.
\end{enumerate}
\end{lemma}
\begin{proof}
Assume that (a) holds. If $G$ is not unipotent, there exists a nontrivial
torus $T$ among the closed subgroups of $G$. The
action of $T$ on $G$ by left multiplica\-tion then gives a fixed point
free action of $T$ on an affine space, which is impossible by \cite[Thm. 1]{Bi1966}.
This contradiction
proves (a)$\Rightarrow$(b).

Conversely, (a) follows from (b) by \cite[Cor. p. 5-02]{Gr1958}.
\quad $\square$ \end{proof}

\begin{theorem}\label{thm13} Let $G$ be a connected affine
algebraic group,
and let ${\mathcal R}G$ be its radical.
\begin{enumerate}[\hskip 4.2mm\rm(a)]
\item If $G$ is solvable, then $G$ is flattenable.
\item If $G$ is flattenable and nonsolvable, then ${\mathcal R}G$ is
not unipotent.
\end{enumerate}
\end{theorem}
\begin{proof} Let $G$ be solvable. Then $G$, as a variety, is isomorphic
to the comp\-le\-ment of a union of several coordinate hyperplanes in some
affine space \cite[Cor. p. 5-02]{Gr1958}; whence (a).

Assume that $G$ is flattenable and nonsolvable. The latter implies that $G$
is not unipotent, hence, by Lemma \ref{afsp}, as a variety, $G$ is not
isomorphic to an affine space. Lemma \ref{thm12} then implies
that there is a nonconstant
invertible function  $f\in k[G]$. By \cite[Thm. 3]{Ro1961}, the map
$G\to {\mathrm {GL}}_1$, $g\mapsto f(g)/f(e)$, is then a nontrivial
character. According to \cite[Lem. 1.1]{Po2011}, the
existence of such a character is equivalent to the property that
${\mathcal R}G$ is not unipotent; whence (b).
\quad $\square$ \end{proof}

\begin{corollary}\label{fre} Let $G$ be a nontrivial connected reductive
algebraic group. If $G$ is flattenable, then the dimension of its center
is positive. In particular, every semisimple $G$ is not flattenable.
\end{corollary}

\subsection*{7. Answering
(LZCP)}

\begin{theorem}\label{aLZCP} There are affine varieties $X$
and $Y$ such that
\begin{enumerate}[\hskip 4.2mm\rm(a)]
\item $X$ is not flattenable;
\item $Y$ and $X\times Y$ are flattenable.
\end{enumerate}
\end{theorem}
\begin{proof} As a variety, any ${\mathrm {SL}}_n$ for $n>1$ is
not flattenable by Corollary \ref{fre}. On the other hand, being open in
the affine space ${\rm Mat}_{m\times m}$, any ${\mathrm {GL}}_m$
is flattenable. The
morphism
\begin{align}\label{om}
{\mathrm {SL}}_n\times {\mathrm {GL}}_1\to {\mathrm {GL}}_n,\quad
(s, a)\mapsto s\,{\rm diag}(a,1,\ldots,1),
\end{align}
is an isomorphism of varieties:
its inverse is
\begin{align*}
{\mathrm {GL}}_n\to{\mathrm {SL}}_n\times {\mathrm {GL}}_1,\quad
g\mapsto \big(g\,{\rm diag}(1/\!\det(g),1,\ldots,1), \det(g)).
\end{align*}
Hence we can take $X={\mathrm {SL}}_n$ for $n>1$, and $Y={\mathrm {GL}}_1$.
\quad $\square$ \end{proof}

In Remark \ref{other} below one can find other examples.

\subsection*{8. Equivariantly Flattenable Varieties}

\begin{definition}\label{def2}
A variety $X$ endowed with an action of an algebraic group $G$
is called
{\it equivariantly} \lb respectively, {\it linearly equivariantly}{\rb}
{\it flattenable}  if there are
\begin{enumerate}[\hskip 4.2mm $\cdot$]
\item an action \lb respectively, a linear action\rb\;of $G$ on
some ${\mathbf A}^{\hskip -.4mm n}$;
    \item a $G$-equivariant open embedding $X\hookrightarrow
    {\mathbf A}^{\hskip -.4mm n}$.
\end{enumerate}
\end{definition}
\begin{definition}\label{def3}  An algebraic group $G$ is called
{\it equivariantly} {\lb}respectively, {\it linearly equivariantly}{\rb}
{\it flattenable} if $G$, as a variety endowed with the $G$-action by left
multiplication, is {\it equivariantly} {\lb}respectively, {\it linearly
equiva\-ri\-antly}{\rb} {\it flattenable}.
\end{definition}

\begin{example}\label{EEE}
Every ${\mathrm {GL}}_n$ is linearly equivariantly flattenable
since ${\mathrm {GL}}_n$ is an invariant open set of
${\rm Mat}_{n\times n}$ endowed with the ${\mathrm {GL}}_n$-action
by left multiplica\-tion.
\end{example}
%%\vskip 2mm

\begin{example}
Every (connected) unipotent affine algebraic group $G$ is, as a variety,
isomorphic to an affine space; hence $G$ is equivariantly flattenable.
In fact, a more general statement,Theorem  \ref{LEF} below, holds.
It is easily seen that
$G$ is linearly equivariantly flattenable only if it is trivial. On the
other hand, the example of a Borel subgroup of ${\mathrm {SL}}_2$
naturally acting on $k^2$ shows that there are nontrivial solvable
linearly equivariantly flattenable groups.
\end{example}

\begin{example}
If the $G_1,\ldots, G_m$ are equivariantly (respectively,
linearly
equiva\-ri\-antly) flattenable groups, then, clearly,
$G_1\times\cdots\times  G_m$ is equivariantly (respecti\-vely, linearly
equivariantly) flattenable as well. In particular, the group
\begin{equation}\label{prod{GL}}
{\mathrm {GL}}_{n_1}\times\cdots\times {\mathrm {GL}}_{n_s}
\end{equation}
 is linearly
equivariantly flattenable for any $n_1,\ldots, n_s$. Taking
$n_1=\ldots=n_s=1$ yields that every
affine algebraic torus is linearly
equivariantly flattenable.
\end{example}

\begin{example}
Generalizing Example 1,
let $A$ be a finite-dimensional asso\-ciative $k$-algebra
with identity. The group $A^\ast$
is a connected affine algebraic group. It is open in $A$
and invariant with respect to the action of $A^\ast$ on $A$
by left multiplication, cf. \cite[I,1.6(9)]{Bo1991}.
Hence $A^\ast$ is a linearly
equivariantly
 flattenable group. For $A={\rm Mat}_{n\times n}$, we obtain
$A^*={\mathrm {GL}}_n$. More generally, if $A$ is
semisimple, then $A^\ast$ is a
group of type \eqref{prod{GL}}, and all groups of type
\eqref{prod{GL}} are obtained in this way.
\end{example}

\begin{example}
 Every $G={\rm SL}_n\times {\rm GL}_1$ is equivariantly
 flattenable.\;Indeed, consider the $G$-module structure on
 $V={\rm Mat}_{n\times n}$ defined by the formula
\begin{equation*}
G\times V\to V,\quad ((s, a), x)\mapsto s\,x\,{\rm diag}(a, 1,\ldots, 1).
\end{equation*}
For $x={\rm diag}(1,\ldots, 1)$, the orbit map $G\to V$,
$g\mapsto g\cdot x$ is then the $G$-equivariant open embedding
\eqref{om}.
\end{example}

\begin{lemma}\label{lfli} Let $X$ be an algebraic variety endowed with an action of an algebraic group $G$.
\begin{enumerate}[\hskip 4.2mm \rm(a)]
\item If $X^G$ is reducible, then $X$ is not linearly equivariantly flattenable.
\item If $X$ is not linearly equivariantly flat\-ten\-able, but equivariantly flat\-ten\-able, then the action $\alpha$ of $G$ on
    %%affine space
    ${\mathbf A}^n$, extending that on $X$, is non\-linea\-rizable.
\end{enumerate}
\end{lemma}

\begin{proof} (a) Arguing on the contrary, assume
 that $X$ is linearly equivariantly flat\-tenable.
 Then there are a linear action of $G$ on
 %%affine space
 ${\mathbf A}^n$ and  a
 $G$-equivariant open embedding $X\hookrightarrow {\mathbf A}^n$. We
 identify $X$ with its image. Then
 $X^G=X\cap ({\mathbf A}^n)^G$. Since $({\mathbf A}^n)^G$ is a linear subspace of ${\mathbf A}^n$,
 and $X$ is open in ${\mathbf A}^n$, this implies that $X^G$ is irreducible, contrary to the assumption. This proves (a).
 %%The proof of (b) is clear.

 (b) Arguing on the contrary, assume that $\alpha$ is linearizable, i.e., there is a linear action $\beta$ of $G$ on ${\mathbf A}^n$ and a $G$-equivariant automorphism
 $\varphi\colon {\mathbf A}^n\to {\mathbf A}^n$, where the left (resp. the right) ${\mathbf A}^n$ is endowed
 with the action $\alpha$ (resp., $\beta$). Then the equivariant open embedding of $X$ in the left ${\mathbf A}^n$ composed with  $\varphi$ is an equivariant open embedding of $X$ in the right ${\mathbf A}^n$. Hence $X$ is linearly equivariantly flattenable, contrary to the assumption. This proves (b).
\quad $\square$ \end{proof}

\begin{theorem}\label{RLEF} The following properties of
a   connected reductive algebraic group
$G$ are equivalent:
\begin{enumerate}[\hskip 4.2mm \rm(a)]
\item $G$ is equivariantly flattenable;
\item $G$ is linearly equivariantly flattenable.
\end{enumerate}
\end{theorem}
\begin{proof} Let $G$ be equivariantly flattenable. By
Definitions \ref{def2},  \ref{def3}, we may (and shall)
identify $G$ with an open orbit in some
${\mathbf A}^{\hskip -.4mm n}$ endowed with a regular
action of $G$. Openness of this orbit implies
$k[{\mathbf A}^{\hskip -.4mm n}]^G=k$. Hence, by
\cite[p. 98, Cor. 2]{Lu1973} (see also \cite[Thm. 6.7]{PV1994}),
there are a closed reductive subgroup $L$ of $G$ and
a finite-dimensional algebraic $L$-module $V$ such that
${\mathbf A}^{\hskip -.4mm n}$ and $G\times^LV$
are $G$-equivariantly isomorphic.
We claim that this implies
\begin{equation}\label{H=G}
L=G.
\end{equation}
If \eqref{H=G} is proved, then ${\mathbf A}^{\hskip -.4mm n}$
and $V$ are $G$-equivariantly isomorphic, which
proves (a)$\Rightarrow$(b).

So it remains to prove \eqref{H=G}. In view of
connectedness of  $G$, to this end it suffices to prove the equality $\dim(L)=\dim(G)$.

Since ${\rm char}(k)=0$, by the Lefschetz principle
we may (and shall) assume that $k=\mathbf C$; in the
remainder of the proof topological terms are related to
the Hausdorff $\mathbf C$-topology. Since
${\mathbf A}^{\hskip -.4mm n}$ is simply connected,
$G/L$ is simply connected as well, hence $L$ is connected.

We now note that the dimension any connected complex reductive
algeb\-raic group $R$ is equal to the maximum $m_R^{\ }$ of $i$
such that $H_i(R)\neq 0$ (singular homology with complex
coefficients). Indeed, if $K$ is a maximal compact sub\-group
of $R$, then the Iwasawa decomposition of $R$ shows
that $R$, as a manifold, is a product of $K$ and a Euclidean
space. Hence $R$ and $K$ have the same homology.
Since $K$ is a compact oriented manifold, this shows
that $m_R^{\ }$ is equal to the dimension of the Lie
group $K$. As $R$ is the complexification of $K$,
the statement follows.

So to prove \eqref{H=G} is the same as to prove
$m_G^{\ }=m_L^{\ }$. In fact,
since $L$ is a subgroup of $G$,
the above equality $m_R^{\ }=\dim(R)$ yields
$m_G^{\ }\geqslant m_L^{\ }$, so to prove  $m_G^{\ }=m_L^{\ }$
we only need to prove the inequality
\begin{equation}\label{m}
m_G^{\ }\leqslant m_L^{\ }.
\end{equation}

As is known (see, e.g., \cite[Chap. IX, Thm. 11.1]{Hu1959}),
the spectral sequence of the natural fiber bundle $G\to G/L$
yields the following inequality for the Betti numbers
\begin{equation}\label{spectr}
\dim_{\mathbf C}(H_{m_G^{\ }}(G))\leqslant
\sum_{i+j=m_G^{\ }} \dim_{\mathbf C}(H_i(G/L))
\dim_{\mathbf C}(H_j(L)).
\end{equation}
On the other hand, since $G\times^LV$ is a vector bundle
over $G/L$ and $G\times^LV$ is isomorphic to
${\mathbf A}^{\hskip -.4mm n}$, we have
\begin{equation}\label{has}
 \dim_{\mathbf C}(H_i(G/L))=
  \dim_{\mathbf C}(H_i({\mathbf A}^{\hskip -.4mm n}))
  =\begin{cases}
 1&\mbox{for $i=0$},\\
 0&\mbox{for $i>0$}.
 \end{cases}
 \end{equation}
 From \eqref{spectr}, \eqref{has} we infer that
 \begin{equation}\label{<}
 0< \dim_{\mathbf C}(H_{m_G^{\ }}(G))\leqslant
 \dim_{\mathbf C}(H_{m_G^{\ }}(L)).
 \end{equation}
 The definition of $m_L^{\ }$ and \eqref{<} then
 yield \eqref{m}. This completes the proof.
\quad $\square$ \end{proof}

\begin{remark}
Using the same argument, but (in the spirit of  \cite{Bo1985})
\'etale cohomology in place of singular homology,
one can avoid applying the Lefschetz principle and
adapt the above proof to the case of  base field of
arbitrary characteristic.
\end{remark}

\begin{theorem}\label{LEF} Let $G$ be a connected
affine algebraic group and let
$L$ be a Levi subgroup of $G$. If $L$ is equivariantly
flattenable, then so is $G$.
\end{theorem}
\begin{proof}
Let $L$ be equivariantly flattenable. Then by
Theorem \ref{RLEF},
 we may (and shall) assume that $L$, as a variety with the
 $L$-action by left multiplication, is an open orbit
 $\mathcal O$ of an algebraic $L$-module $V$. This
 implies that $\dim(L)=\dim(V)$, and therefore,
\begin{equation}\label{ST}
\dim(G\times^L V)=\dim(G)-\dim(L)+\dim(V)=\dim(G).
\end{equation}

We identify $V$ with the fiber of $G\times^LV\to G/L$
over the point of $G/L$ corresponding to $L$. Since the
$L$-stabilizer of any point $v\in \mathcal O$ is trivial,
the $G$-stabilizer
of $v$ is trivial as well. From this  and \eqref{ST} we
infer that $G\to G\times^L V$, $g\mapsto g\cdot v$,
is a $G$-equivariant (with respect to the action of $G$
on itself by left multiplication) open embedding. Now
we note that
$G/L$ is isomorphic to the underlying variety of ${\mathcal R}_uG$,
 therefore, by Lemma  \ref{afsp}, to an affine space. Since,
 by  Quillen--Suslin \cite{Qu1976}, \cite{Su1976},
 algebraic vector bundles over affine spaces are trivial, we
conclude that  the variety $G\times^L V$ is isomorphic
to an affine space. This completes the proof.
\quad $\square$ \end{proof}

\begin{corollary}\label{SEF} Every connected
solvable affine algebraic group is equiva\-riantly flattenable.
\end{corollary}

\begin{proof}  Levi subgroups of connected solvable
affine algebraic groups are tori. As the latter are
equivariantly flattenable,
the claim follows from Theorem \ref{LEF}.
\quad $\square$ \end{proof}

\subsection*{9.  Equivariant Flattenability vs. Linear Equivariant Flattenability}
%%Flattenable vs. Equivariantly Flattenable Varieties}

The following shows that there are
affine equivariantly flattenable varieties
%%endowed with actions of reductive algebraic groups,
which are not linearly equivariantly flat\-ten\-able.

\begin{example}
As is known (see, e.g., references in survey \cite{Kr1996}),
there are affine spaces endowed with nonlinea\-riz\-able actions
of reductive algebraic groups. So they are equivariantly flattenable, but,
by Lemma  \ref{thm12},
%%they
%%are
not linearly equivariantly flat\-ten\-able.
\end{example}

\begin{question}
Are there flattenable reductive algebraic groups, which are
not equi\-va\-riantly flattenable?
\end{question}

\subsection*{10. Equivariant Flattenability vs. Linearizability}

In this section, we construct a series of affine flattenable varieties
%%$X$
endowed with actions of finite groups,
%% $G$,
which are not linearly equivariantly flattenable.
%%By Lemma \ref{lfli},
%%every such
%%%%$X$
%%variety yields either an example of flattenable, but not equivariantly flattenable
%%variety or an example of a nonlinearizable action of a finite group on affine space.

%%for every such $X$, either $X$ is not equivariantly flattenable or the action of $G$ on an %%affine space extending that on $X$ is nonlinearizable.

%%\begin{enumerate}[\hskip 5.2mm ---]
%%\item $X$ is not linearly equivariantly flattenable;
%%\item if $X$ is equivariantly flattenable, then the action of $G$ on an affine space, %%extending that on $X$б is nonlinearizable.
%%\end{enumerate}

%%It is an intriguing problem to understand which of these possibilities holds for every %%concrete $X$ and whether
%%every of this possibilities is realizable.

%%the following alternative holds:
%%\begin{enumerate}[\hskip 0mm $\cdot$ ]
%%\item either $X$ is not equivariantly flattenable,
%%\item or $X$ is equivariantly flattenable, but
%%the action of $G$ on an affine space, extending that on $X$,
%%is nonlinearizable.
%%\end{enumerate}

The construction is as follows.

%%\begin{example}\label{sf}
%%Let
Take a pair $S$, $G$, where
\begin{enumerate}[\hskip 5.6mm \rm(A)]

\item[$({\rm c}_1)$] $S$ is a connected semisimple
algebraic group;
%%and let
\item[$({\rm c}_2)$] $G$ is a finite subgroup of $S$ such that
${\mathcal Z}_S(G)$
(the   centralizer of $G$
in $S$) is finite and nontrivial.
\end{enumerate}

Note that such pairs $S, G$ do exist.

\begin{example} Let $F$ be
a nontrivial finite group, satisfying the conditions:
\begin{enumerate}[\hskip 4.2mm \rm(i)]
 \item there are no nontrivial characters $F\to {\rm GL}_1$;
 \item there is a faithful irreducible representation
 $\varphi\colon F\to {\rm GL}_n$ for some\;$n$.
 \end{enumerate}
 For instance, (i) and (ii) hold for any nontrivial simple $F$.\;By
(i), we have
 $\varphi(F)\subset {\rm SL}_n$, and by (ii) we may
 (and shall) iden\-ti\-fy $F$ with $\varphi(F)$; thus
 %%consider
 $F$
 %%as
 is
 a nontrivial subgroup of ${\rm SL}_n$ whose natural linear action on $k^n$ is irreducible.\;By
 %%(b) and
 Schur's lemma, ${\mathcal Z}_{{\rm SL}_n}(F)$
%% the centralizer
%% of $F$ in ${\rm SL}_n$
is the cyclic group
 $\{{\varepsilon}I_n\mid \varepsilon\in k^\ast, \varepsilon^n=1\}$
 of order $n$.
 %%, which is the center of  ${\rm SL}_n$.\;Since $F$ is nontrivial,
%% %%We
 Since $F$ is nontrivial, $n\geqslant 2$.
%% %%  be\-cau\-se of (a),
%% so this center is
%% nontrivial.\;
Hence $S={\rm SL}_n$, $G=F$ is
 the pair of interest.
\end{example}

\begin{example}\label{re} (See \cite{Bo1961}, \cite{CS1987}, \cite{Se1999}). Let $S$ be a connected semisimple algebraic group. Let $T$ and $N_S(T)$ be resp. a maximal torus of $S$ and its normalizer in $S$.  Assume that the Weyl group $N_S(T)/T$ contains $-1$; if $S$ is simple, this is equivalent to the condition that $S$ is {\it not} of type ${\sf A}_\ell$ for $\ell\geqslant 2$,
${\sf D}_\ell$ for odd $\ell$, or ${\sf E}_6$ (see \cite[Tabl. I--IX]{Bou1968}).
%%of type ${\sf A_1}$, ${\sf B}_\ell$, ${\sf C}_\ell$, ${\sf D}_\ell$ for even $\ell$, ${\sf %%G}_2$, ${\sf F}_4$, ${\sf E}_7$, ${\sf E}_8$.
Let $n$ be an element of $N_S(T)$ representing $-1$ of $N_S(T)/T$. Then either $n^2=1$ (this is so if $S$ is adjoint)
or $n^2$ is an element of order 2 of the center of $S$. Assume that $n^2=1$ and let $G$ be the
subgroup of $S$ generated by $n$ and all the elements of $T$ whose order is 2. Then $G$ is an elementary Abelian group of order $2^{{\rm rk}(S)+1}$, for which condition $({\rm c}_2)$ holds.
\end{example}

We return back to the description of the construction. By $({\rm c}_2)$, there is a non\-iden\-tity element $z\in {\mathcal Z}_S(G)$, and its order is finite.  In view of ${\rm char}\,(k)=0$, the finiteness of the order entails that $z$ is semisimple. Hence (see \cite[11.10]{Bo1991}) the element $z$ lies is a maximal torus of $S$. Since, in turn, any torus of $S$ lies in a Borel subgroup of $S$ (cf. \cite[11.3]{Bo1991}), we infer that there is a Borel subgroup $B$ of $S$ containing
two different elements of ${\mathcal Z}_S(G)$:
\begin{equation}\label{bcz}
1, z\in B\cap {\mathcal Z}_S(G),\quad z\neq 1.
\end{equation}

 Let $B^-$ be the  Borel subgroup of $S$ opposite to $B$. Then the ``big cell''
 $\Theta:=B^-B$ is an open subset of $S$ isomorphic
 to the comp\-lement of a union of several coordinate
 hyperplanes in $L:={\mathbf A}^{\hskip -.4mm \dim(S)}$;
 in particular, $\Theta$ is an affine flattenable variety. In view of \eqref{bcz}, we have
\begin{equation}\label{tcz}
1, z\in \Theta\cap {\mathcal Z}_S(G),\quad z\neq 1.
\end{equation}

Now we consider the conjugating action of $G$ on
 $S$.\;Its
 fixed point set $S^G$ is ${\mathcal Z}_{S}(G)$.\;The
 variety
 \begin{equation}
 \label{vvX}
 X:=\bigcap_{g\in G} g\Theta g^{-1}
 \end{equation}
 is a $G$-stable open subset of
 $S$.\; Since  $\Theta$ is an affine flattenable variety, $X$
 is such a variety, too.\;From \eqref{tcz}, \eqref{vvX}, and the finiteness of
 ${\mathcal Z}_{S}(G)$ we conclude that
 \begin{equation}\label{cp}
X^G=X\cap S^G\;\;\mbox{is a finite set containing two different points $1$ and $z$}.
 \end{equation}

 \begin{proposition} The affine flattenable $G$-variety $X$ defined by formula {\rm \eqref{vvX}} shares the following properties:
 \begin{enumerate}[\hskip 3.2mm\rm(i)]
 \item $X$ is not linearly equivariantly flattenable.
 \item %%The following alternative holds:
 The alternative holds:
 \begin{enumerate}[\hskip 0mm ---]
 \item ei\-ther $X$
is not equivariantly flattenable,
\item or $X$ is equivariantly flattenable, but
the action of $G$ on an affine space, extending that on $X$,
is nonlinearizable.
 \end{enumerate}
 \end{enumerate}
 \end{proposition}

 \begin{proof} %%Since $S$ is semisimple, its center ${\mathcal C}S$ is finite.  By condition %%$({\rm c}_1)$, this center is non\-trivial. Therefore, equality \eqref{fp} and the finiteness %%of ${\mathcal C}S$
 In view of \eqref{cp}, the variety
 %%yield that
 $X^G$ is a reducible.
 %% variety.
 By Lemma \ref{lfli}(a),
 %%the reducibility of this variety
 this implies (i). In turn, (i) and Lemma \ref{lfli}(b) imply (ii).
 %%(i) Arguing on the contrary, assume
 %%that $X$ is linearly equivariantly flattenable.
 %%Then there is a linear action of $G$ on $L$ such that there is a
 %%$G$-equivariant open embedding $X\hookrightarrow L$; we
 %%identify $X$ with the image of this embedding.\;Hence
 %%$X^G=X\cap L^G$. Since $L^G$ is a linear subspace of $L$,
 %%and $X$ is open in $L$, this implies that $X^G$ is irreducible.\;The
 %%latter contradicts \eqref{fp}, because $S$ is semisimple,
 %%and hence ${\mathcal C}S$ is finite (and nontrivial). This proves (i).
 %%
 %%Statement (ii) follows from (i).
 \quad $\square$
 \end{proof}

 Thus there are two a priori possibilities: our construction yields either an example of an action of a finite group on a flattenable variety, which is not equi\-vari\-antly flattenable,
or an ex\-ample of a nonlinearizable action of a finite group on an affine space.
The linearization problem for reductive group actions on affine spaces has received much attention in the literature, but,
to the best of my knowledge, at this writing (October 2019)
the existence of nonlinearizable actions of finite Abelian groups on affine spaces remains an unsolved problem.
Being linked with this subject, our construction leads to the intriguing ques\-ti\-ons:
%%This leads to the intriguing questions:

\begin{question} Is each of these two a priori possibilities realizable for an
appro\-pri\-ate pair $S$, $G$?
%%find out whether
%%there
%%is there a pair $S$, $G$, for which this possibility holds?
%% which realizes it.
%%it is
%%realiz\-able by
%%for
%%implemented by an appropriate pare $S$, $G$.
%%implementation of our construction.
\end{question}

\begin{question} Take $S={\rm PGL}_2$, and let $G$ be the group from Example \ref{re}. This $G$ is the Klein four-group, it is  isomorphic to $\mathbb Z/2\oplus\mathbb Z/2$. The variety $X$ is isomorphic to an affine open subset of ${\mathbf A}\!^3$. Which of the two a priori possibilities indicated above for the action of $G$ on $X$ is realized in this case?
%%Find out which of these two possibilities holds for each specific implementation of our %%construction.
\end{question}
%%\end{example}

%%\begin{remark} For $X$ in Example \ref{sf}, as a $G$-variety,
%%the following alter\-na\-tive holds.
%%Ei\-ther $X$
%%is not equivariantly flattenable.\;Or $X$ is equivariantly flattenable, but
%%the action of $G$ on an affine space, extending that on $X$,
%%is nonlinearizable.
%%\end{remark}

%%\begin{question}
%%Are there flattenable reductive algebraic groups, which are
%%not equi\-va\-riantly flattenable?
%%\end{question}

\subsection*{11. Equivariantly Flattenable Subgroups
of the Cremona Groups}

As an application, below is briefly surveyed
a special role of equivariantly flat\-ten\-able subgroups in
the conjugacy problem for algebraic subgroups of the
Cremona groups ${\rm Cr}_n$
(i.e., in the rational linearization problem). We refer to
\cite{Po2013_1}, \cite{PV1994} and  references therein
regarding the basic definitions and properties of rational
algebraic group actions and, in particular, the definition of embeddings
\begin{equation}\label{Cremona}
{\rm Cr}_1\subset {\rm Cr}_{2}\subset\cdots\subset
{\rm Cr}_n\subset {\rm Cr}_{n+1}\subset\cdots.
\end{equation}

\begin{theorem}[{{\rm \cite[Thm. 2.1]{Po2013_1}}}]\label{rconj}
Let $G$ be a connected algebraic sub\-group of
the Cremona group ${\rm Cr}_n$. Assume that
\begin{enumerate}[\hskip 2.2mm \rm(i)]
\item $G$ is linearly equivariantly flattenable;
\item the natural rational action of $G$ on ${\mathbf A}^{\hskip -.4mm n}$
is locally free.
\end{enumerate}
If the field extension $k({\mathbf A}^{\hskip -.4mm n})^G/k$
is purely transcendental, then $G$ is conjugate in  ${\rm Cr}_n$
to a subgroup of
${\mathrm {GL}}_n$ \lb i.e., the natural rational action of $G$ on
${\mathbf A}^{\hskip -.4mm n}$ is rationally linearizable\rb.
\end{theorem}

For tori, the sufficient condition of Theorem \ref{rconj}
for rational linearization is also necessary:

\begin{theorem}[{{\rm \cite[Thm. 2.4]{Po2013_1}}}] Let $T$
be an affine algebraic torus in the Cre\-mo\-na group ${\rm Cr}_n$.
The following properties are equivalent:
\begin{enumerate}[\hskip 4.2mm \rm(a)]
\item $T$ is conjugate in V to a subgroup of ${\mathrm {GL}}_n$;
\item the field extension $k({\mathbf A}^{\hskip -.4mm n})^T/k$
is purely transcendental.
\end{enumerate}
 \end{theorem}

 The sufficient condition of Theorem \ref{rconj} for rational
 linearization always holds true in the stable range:

 \begin{theorem}[{{\rm \cite[Lem. 2.2 \& Thm. 2.2]{Po2013_1}}}] \label{strl}
 Let $G$ be a connected algebraic subgroup of the Cremona group ${\rm Cr}_n$
 such that assumptions {\rm (i)}, {\rm (ii)} of Theorem {\rm\ref{rconj}} hold.
 Then there is an integer $s\geqslant 0$ such that, for the rational action of
 $G$ on ${\mathbf A}^{\hskip -.4mm n+s}$ determined by
 embedding  {\rm\eqref{Cremona}},
 the field extension
 $k({\mathbf A}^{\hskip -.4mm n+s})^G/k$ is purely transcendental.
 \end{theorem}

 \begin{remark} By
  \cite[Thm. 2.6]{Po2013_1}, if $G$ in  Theorem \ref{strl} is
  a torus, then  one can take $s=\dim(G)$.
  \end{remark}

 \begin{remark}
 In general,  $s$ in Theorem \ref{strl} is strictly positive.
 For example, by
  \cite[Cor. 2.5]{Po2013_1}, the Cremona group ${\rm Cr}_n$ for
  $n\geqslant 5$ contains an $(n-3)$-dimensional affine
  algebraic torus, which is not conjugate in ${\rm Cr}_n$
  to a subgroup of
  ${\mathrm {GL}}_n$.
 \end{remark}

\subsection*{12. Equivariantly Flattenable Groups and
Special Groups in the Sense of Serre}

Recall from \cite[4.1]{Se1958}
that an algebraic group $G$ is called {\it special} if   every
principal $G$-bundle (which, by definition
\cite[2.2]{Se1958}, is locally trivial in \'etale topology)
is locally trivial in the Zariski topology. By
\cite[Sect. 4.1, Thm. 1]{Se1958} special group
is automatically connected and affine. Special
groups are classified:
\begin{theorem}
\label{classp}
The following properties of a connected affine algebraic
group $G$ are equivalent:
\begin{enumerate}[\hskip 4.2mm\rm(a)]
\item $G$ is special;
\item maximal connected semisimple subgroups of $G$
are isomorphic to
a group of type
\begin{equation*}
{\mathrm {SL}}_{n_1}\times\cdots
\times {\mathrm {SL}}_{n_r}\times {\rm Sp}_{2m_1}\times
\cdots \times {\rm Sp}_{2m_s}.
\end{equation*}
\end{enumerate}
\end{theorem}
\begin{proof} The implications
(a)$\Rightarrow$(b) and (b)$\Rightarrow$(a) are
proved respectively in \cite{Gr1958} and
\cite{Se1958}.
\quad $\square$ \end{proof}

In \cite[Lem. 2.2]{Po2013_1} is sketched a reduction
of the following claim to  \cite[Thm. 1.4.3]{Po1994}.
Below is given the complete self-contained argument.

\begin{theorem}
\label{flsp}
Every linearly equivariantly flattenable affine
algebraic group $G$ is special.
\end{theorem}
\begin{proof}
By Definitions \ref{def2}, \ref{def3}, there is
a finite-dimensional algebraic $G$-module $V$ with
a $G$-orbit $\mathcal O$ such that
\begin{enumerate}[\hskip 2.2mm\rm(i)]
\item\label{ope1} $\mathcal O$ is open in $V$;
\item\label{tri1} the points of $\mathcal O$
have trivial $G$-stabilizers.
\end{enumerate}
By \eqref{tri1}, $V$  is faithful; hence we may
(and shall) identify $G$ with a closed subgroup of ${\mathrm {GL}}(V)$.

Next, consider the $i$th direct summand of
$$L:=V\oplus\cdots\oplus V
\quad \mbox{($n:=\dim(V)$ copies)},
$$
as a linear subspace $V_i$ of $L$, denote by $\pi_i$
the natural projection $L\to V_i$, and identify $V$ with $V_1$.
The ${\mathrm {GL}}(V)$-module $L$ contains
a ${\mathrm {GL}}(V)$-orbit $\widetilde{\mathcal O}$
such that
\begin{enumerate}[\hskip 3.2mm\rm(i)]
\item[\rm(iii)]\label{ope2} $\widetilde{\mathcal O}$ is open in $L$;
\item[\rm(iv)]\label{tri2} the points of $\widetilde{\mathcal O}$
have trivial ${\mathrm {GL}}(V)$-stabilizers.
\end{enumerate}

From \eqref{ope1}, (iii) we infer that
$\mathcal O\cap \pi_1(\widetilde{\mathcal O})\neq\varnothing$
and $\pi_1^{-1}(\mathcal O)\cap \widetilde{\mathcal O}$
is an nonempty open subset
of $\widetilde{\mathcal O}$.
Take a point
\begin{equation}\label{a}
a\in \mathcal O\cap \pi(\widetilde{\mathcal O})
\end{equation}
and consider in $L$ the affine subspace
\begin{equation}\label{A}
A:=\{a+v_2+\cdots +v_{n}\in L\mid v_i\in V_i  \mbox{ for all $i$}\}.
\end{equation}

From (iii), \eqref{a}, \eqref{A} we deduce that
$A\cap \widetilde{\mathcal O}$ is a nonempty
open subset of $A$, and from \eqref{tri1} that the
$G$-orbit of every point of
$\pi_1^{-1}(\mathcal O)\cap \widetilde{\mathcal O}$
intersects $A\cap \widetilde{\mathcal O}$ at a single point.
This means that the natural action of  $G$ on $\widetilde{\mathcal O}$
admits a rational section. In view of (iv), this,
in turn, means that the natural map
${\mathrm {GL}}(V)\to {\mathrm {GL}}(V)/G$ admits
a rational section. Since the group ${\mathrm {GL}}(V)$
is special, this implies, according to \cite[Sect. 4.3, Thm. 2]{Se1958},
that the group $G$ is special as well.
\quad $\square$ \end{proof}
\begin{remark}\label{ratio} Since $A$ is rational and
$A\cap \widetilde{\mathcal O}$ is open in $A$,  the above
proof of Theorem \ref{flsp} shows that
${\mathrm {GL}}(V)/G$ is a rational variety.
\end{remark}

\subsection*{13. Classifying equivariantly flattenable groups}

Theorem \ref{LEF} natural\-ly leads to the following
\begin{problem}\label{classif} Obtain a classification
of equivariantly flattenable
reductive algebra\-ic groups.
\end{problem}

\begin{theorem} For every nontrivial equivariantly flattenable
reductive algeb\-ra\-ic group $G$, the following properties hold:
\begin{enumerate}[\hskip 4.2mm\rm(a)]
\item the derived group $L$ of
$G$ is a semisimple group of type
\begin{equation*}
{\mathrm {SL}}_{n_1}\times\cdots
\times {\mathrm {SL}}_{n_r}\times {\rm Sp}_{2m_1}\times
\cdots \times {\rm Sp}_{2m_s};
\end{equation*}
\item the radical of $G$ is a central torus $C$ of positive dimension;
\item $G=L\cdot C$ and $C\cap L$ is finite.
\end{enumerate}
\end{theorem}
\begin{proof} Since $L$ is a maximal connected
semisimple subgroup of $G$,
combining Theorems \ref{flsp} and \ref{classp} yields (a).
Combining Corollary \ref{fre}
and \cite[14.2, Prop.]{Bo1991}
implies (b) and (c).
\quad $\square$ \end{proof}

Problem \ref{classif} looks manageable. Initially, being
influenced by Example \ref{EEE}.4,  the author
was
even
overoptimistic and put forward the conjecture that
all equivariantly flattenable reductive algebraic groups
are that of type \eqref{prod{GL}} (see \cite[p. 221]{Po2013_1});
this overoptimism was shared by some of the participants of the 2013
Oberwolfach meeting on algebraic groups who even sketched
a plan of possible proof. So it  came as a surprise, when
in \cite{BGM2017} the examples of equivariantly flattenable
reductive groups of a type distinct from \eqref{prod{GL}}
have been revealed, making Problem \ref{classif} even
more intriguing. Below we describe them.

\begin{theorem}[{{\cite[2.1]{BGM2017}}}]\label{BGM}
For every positive integer $n$, the group
\begin{equation}\label{G}
G:={\rm Sp}_{2n}\times {\mathrm {GL}}_{2n-1}\times
{\mathrm {GL}}_{2n-2}\times\cdots\times {\mathrm {GL}}_{1}
\end{equation}
is equivariantly flattenable.
\end{theorem}
\begin{proof}[Sketch of proof] Consider the vector space
\begin{equation}\label{V}
V:={\rm Mat}_{2n\times 1}\oplus {\rm Mat}_{2n\times (2n-1)}
\oplus {\rm Mat}_{(2n-1)\times (2n-2)}\oplus\cdots\oplus {\rm Mat}_{2\times 1}
\end{equation}
and the linear action of $G$ on $V$ defined for the elements
\begin{equation*}
\begin{gathered}
g:=(A,  B_{2n-1},  B_{2n-2},\ldots, B_2,  B_1)\in G,
A\in {\rm Sp}_n, B_d\in {\mathrm {GL}}_{d},\\
v:=(X,  Y_{2n-1},  Y_{2n-2},\ldots, Y_{1})\in V, X\in
{\rm Mat}_{2n\times 1},
Y_d\in {\rm Mat}_{(d+1)\times d}
\end{gathered}
\end{equation*}
by the formula
\begin{equation*}
g\cdot v:=(AX,  AY_{2n-1}B_{2n-1}^{\top},  B_{2n-1}
Y_{2n-2} B_{2n-2}^{\top},\ldots, B_{2}Y_{1} B_{1}^{\top}).
\end{equation*}

From \eqref{G} and \eqref{V} we deduce
\begin{equation}\label{dim}
\dim(G)\overset{\eqref{G}}{=\hskip -2mm=}2n^2+n+
\sum_{i=1}^{2n-1}i^2=
2n+\sum_{i=1}^{2n-1}i(i+1)
\overset{\eqref{V}}{=\hskip -2mm=}\dim(V).
\end{equation}

Next, one shows the existence of a point $v_0\in V$ whose
$G$-stabilizer is trivial. In view of \eqref{dim}, the orbit
map $G\to V$, $g\mapsto g\cdot v_0$ is then a
$G$-equivariant open embedding.
\quad $\square$ \end{proof}

\begin{remark}\label{other} Since the group ${\rm Sp}_{2n}$
is not flattenable by Corollary \ref{fre}, but
${\mathrm {GL}}_{2n-1}\times {\mathrm {GL}}_{2n-2}\times
\cdots\times {\mathrm {GL}}_{1}$ is flattenable, Theorem \ref{BGM} provides
other  (than that in the proof of Theorem \ref{aLZCP})
examples, which yield the affirmative answer to (LZCP).
\end{remark}

\subsection*{14. Locally Equivariantly Flattenable Varieties}
Similarly to flatten\-ability, equivariant flattenability admits an evident local ver\-sion:

\begin{definition}[{{\rm \cite[Def. 4(iii)]{Pe2017},
up to change of terminology}}]\label{Greq}
A variety $X$ endowed with an action of an algebraic
group $G$ is called
{\it equivariantly} {\lb}respecti\-vely, {\it linearly
equivariantly}{\rb} {\it locally flattenable}  if  for every
point  $x\in X$ there is an equivariantly {\lb}respectively,
linearly equivariantly{\rb} flattenable $G$-stable open
subset of $X$ containing $x$.
\end{definition}

Definition \ref{Greq} leads to the following
equivariant version of Gromov's ques\-tion\;\eqref{Gr}:
\begin{equation}\label{EGr}
\begin{split}
&\mbox{\it Is every irreducible smooth rational $G$-variety}\\[-1mm]
&\mbox{\it equivariantly locally flattenable? }
\end{split}
\tag{EqGr}
\end{equation}
The examples in \cite[Sect. 4]{Pe2017}  show that the
answer to \eqref{EGr} is ne\-ga\-tive.


\begin{thebibliography}{BGM\,2017}

\bibitem[Ak\,1993]{Ak1993} Akbulut, S., {\it Lectures on algebraic spaces},
in: {\it Algebra and Topology}, 1992 Proceedings of KAIST
mathematics workshop, KAIST, Daejeon, Republic of Korea, 1993, pp.\;1--15.

\bibitem[Bi\,1966]{Bi1966} Bia{\l}ynicki-Birula, A., {\it Remarks on the
action of an algebraic torus on $k^n$}, Bull. Acad. Polon.
Sci, S\'er. sci. math., astr., phys.
    {\bf XIV} (1966),  no. 4, 177--181.


\bibitem[BHSV\,2008]{BHSV2008} Bodn\'ar, G. , Hauser, H., Schicho, J.,
Villamayor U., O., {\it Plain varieties}, Bull. Lond. Math. Soc. {\bf 40} (2008),
no. 6, 965--971.

 \bibitem[BB\,2014]{BB2014}  Bogomolov, F., B\"ohning, C., {\it On uniformly
 rational varietie}, in:  {\it Topology, Geo\-met\-ry, Integrable Systems,
 and Mathematical Physics}, Amer. Math. Soc. Transl. Ser. 2, Vol. 234,
 Adv. Math. Sci. {\bf 67}, Amer. Math. Soc., Providence, RI, 2014, pp. 33--48.

 \bibitem[Bo\,1961]{Bo1961} Borel, A., {\it Sous-groupes commutatifs et torsion des groupes de Lie compacts connexes}, T\^ohoku Math. J. {\bf 13} (1961), 216--240.

      \bibitem[Bo\,1985]{Bo1985} Borel, A., {\it On affine algebraic homogeneous spaces},
      Arch. Math. {\bf 45} (1985), 74--78.

      \bibitem[Bo\,1991]{Bo1991} Borel, A., {\it Linear Algebraic Groups}, 2nd Ed.
      Graduate Text in Mathematics, Vol. 126, Springer-Verlag, New York, 1991.

      \bibitem[Bou\,1968]{Bou1968} Bourbaki N., {\it Groupes et Alg\`ebres de Lie}, Chap. IV, V, VI, Hermann, Paris, 1968.

 \bibitem[BGM\,2017]{BGM2017} Burde, D., Globke, W., Minchenko, A.,
 {\it \'Etale
 representations for reductive algebraic groups with factors ${\rm Sp}_n$
 or ${\rm SO}_n$}, Transformation Groups (2017),
 {\tt doi:10.1007/s00031-018-9483-8}.

\bibitem[Ch\,1954]{Ch1954}     Chevalley, C., {\it On algebraic group varieties},
J. Math. Soc. Japan {\bf 6} (1954), nos. 3--4,
303--324.

     \bibitem[Ch\,1958]{Ch1958} Chevalley, C., {\it Les classes d’\'equivalence
     rationnelle}, {\rm I}, in:  {\it Anneaux de Chow et Applications},
     S\'eminaire Claude Chevalley, 1958, exp. no. 2, pp. 1--14.

\bibitem[CZ\,2017]{CZ2017} Chin, C., Zhang, D.-Q.,
{\it Rationality of homogeneous varieties},
  Trans. Amer. Math. Soc. {\bf 369} (1987), 2651--2673.

  \bibitem[CS\,1987]{CS1987} Cohen, A. M., Seitz, G. M., {\it The $r$-rank of the groups of exceptional Lie type}, Indag. Math. {\bf 49} (1987), 251--259.


\bibitem[Gr\,1989]{Gr1989} Gromov, M., {\it Oka's principle for holomorphic
sections of elliptic bundles},
J. Amer. Math. Soc. {\bf 2} (1989), no. 4, 851--897.

\bibitem[Gro\,1958]{Gr1958}
Grothendieck, A., {\it Torsion homologique et sections rationnelle}, in: {\em
  Anneaux de Chow et Applications}, S\'eminaire Claude Chevalley,
  Vol. 3, Exp. no. 5, Secr\'etariat math., Paris, 1958, pp. 1--29.


\bibitem[Gu\,2015]{Gu2015}    Gupta, N., {\it A survey on Zariski Cancellation Problem},
Indian J. Pure Appl. Math. {\bf 46} (2015), no. 6, 865--877.

\bibitem[Hu\,1959]{Hu1959} Hu, S.-T., {\it Homotopy Theory}, Academic
Press, New York, 1959.


\bibitem[Kr\,1996]{Kr1996} Kraft, H., {\it Challenging problems in affine $n$-space},
Ast\'erisque {\bf 47} (802) (1996), 295--317.

\bibitem[Lu\,1973]{Lu1973} Luna, D., {\it Slices \'etales}, M\'em. SMF {\bf 33} (1973), 81--105.

\bibitem[Ma\,1974]{Ma1974} Manin, Y., {\it Cubic Forms}, North-Holland, Amsterdam
1974.

\bibitem[Pe\,2017]{Pe2017} Petitjean, C., {\it Equivariantly uniformly rational
varieties}, {\tt arXiv:1505.03108v2} (2017).

    \bibitem[Po\,1994]{Po1994}
Popov, V. L., {\it Sections in invariant theory}, in: {\it Proceedings of The
  Sophus Lie Memorial Conference}, Oslo $1992$,
  Scandinavian University Press, Oslo, 1994, pp. 315--362.

\bibitem[Po\,2011]{Po2011}   Popov, V. L., {\it On the Makar-Limanov,
 Derksen invariants, and finite auto\-mor\-phism groups  of algebraic varieties},
  in: {\it Affine Algebraic Geometry: The Rus\-sell Fest\-schrift},
 CRM Proceedings and Lecture Notes, Vol. 54, Amer. Math. Soc., 2011, pp. 289--311,


 \bibitem[Po\,$2013_1$]{Po2013_1}  Popov, V. L., {\it Some subgroups of the Cremona groups},
 in: {\it Affine Algebraic Geo\-met\-ry}, Proceedings (Osaka, Japan, 3–6 March 2011),
 World Scientific, Singapore, 2013, pp. 213--242.

 \bibitem[Po\,$2013_2$]{Po2013_2}   Popov, V. L., {\it Rationality and the FML invariant},
 J. Ramanujan Math. Soc.  {\bf 28A} (2013), 409--415.

     \bibitem[PV\,1994]{PV1994} Popov,  V. L., Vinberg, E. B.,
{\it Invariant theory},
in: {\it Algebraic Geometry} IV, En\-cyc\-lopaedia of  Mathematical Sciences,
Vol. 55, Springer-Verlag, Berlin, 1994, pp. 123--284.

\bibitem[Qu\,1976]{Qu1976} Quillen, D., {\it Projective modules over polynomial
rings}, Invent. Math.
{\bf 36} (1976), no. 1, 167--171.

\bibitem[Ro\,1961]{Ro1961} Rosenlicht, M., {\it Toroidal algebraic groups},
Proc. Amer. Math. Soc. {\bf 12} (1961), 984--988.

\bibitem[Se\,1958]{Se1958}
Serre, J.-P., {\it Espaces fibr\'es alg\'ebriques}, in: {\it Anneaux de Chow et
  Applications, S\'eminaire Claude Chevalley}, Vol. 3, Exp. no. 1,
  Secr\'etariat math\'ematique, Paris, 1958, pp. 1--37.

  \bibitem[Se\,1999]{Se1999}
  Serre, J.-P., {\it Sous-groupes finis des groupes de Lie},
  S\'eminaire Bourbaki, Exp. no. 864 (1998--99), in: Serre, J.-P., {\it Expos\'es de S\'eminaires} 1950--1999. Documents Math\'ematiques, Soc. Math. de France, Paris, 2001, pp. 233--248.

\bibitem[Su\,1976]{Su1976} Suslin, A. A., {\it Projective modules over
polynomial rings are free},
Soviet Math. {\bf 17} (1976), no. 4, 1160--1164.

\bibitem[Ti\,2011]{Ti2011}  Timashev, D. A., {\it Homogeneous Spaces and
Equivariant Embeddings}, En\-cyc\-lo\-paedia of Mathematikcal Sciences,
Vol. 138, Subseries {\it Invariant Theory and Algebraic
Transformation Groups}, Vol. VIII, Springer-Verlag, Berlin, 2011.


\end{thebibliography}
\end{document}